\documentclass[journal]{IEEEtran}

\ifCLASSINFOpdf

\else

\fi

\usepackage{subfigure}
\usepackage{graphicx}          
\usepackage{caption}                               
\usepackage[figurename=Figure]{caption}
\usepackage[dvips]{epsfig}    
\usepackage{latexsym,amssymb}
\usepackage{amsmath, amsbsy, amsthm}
\usepackage{amsopn, amstext}
\usepackage{colortbl}
\usepackage{cite}
\newtheorem{remark}{Remark}
\newtheorem{theorem}{Theorem}
\newtheorem{definition}{Definition}
\newtheorem{lemma}{Lemma}

\newtheorem{assumption}{Assumption}

\newtheorem{proposition}{Proposition}
\allowdisplaybreaks
\begin{document}
%
\title{{\huge Open-loop and Closed-loop Local and Remote Stochastic Nonzero-sum Game with Inconsistent Information Structure }}
%
%
%
%
\author{Xin Li, Qingyuan Qi$^{*}$ and Xinbei Lv
\thanks{This work was supported by National Natural Science Foundation of China under grants 61903210, Natural Science Foundation of Shandong Province under grant ZR2019BF002, China Postdoctoral Science Foundation under grant 2019M652324, 2021T140354, Qingdao Postdoctoral Application Research Project, Major Basic Research of Natural Science Foundation of Shandong Province (ZR2021ZD14).(Corresponding author: Qingyuan Qi.)

X. Li (lx998866@163.com) is with Institute of Complexity Science, College of Automation, Qingdao University, Qingdao, China 266071. Q. Qi (qingyuan.qi@hrbeu.edu.cn) and X. Lv (lvxinbei@hrbeu.edu.cn) are with Qingdao Innovation and Development Center of Harbin Engineering University, Qingdao, China, 266000.
}}
\maketitle



%
\IEEEpeerreviewmaketitle

\begin{abstract}
In this paper, the open-loop and closed-loop local and remote stochastic nonzero-sum game (LRSNG) problem is investigated. Different from previous works, the stochastic nonzero-sum game problem under consideration is essentially a special class of two-person nonzero-sum game problem, in which the information sets accessed by the two players are inconsistent. More specifically, both the local player and the remote player are involved in the system dynamics, and the information sets obtained by the two players are different, and each player is designed to minimize its own cost function. For the considered LRSNG problem, both the open-loop and closed-loop Nash equilibrium are derived. The contributions of this paper are given as follows. Firstly, the open-loop optimal Nash equilibrium is derived, which is determined in terms of the solution to the forward and backward stochastic difference equations (FBSDEs). Furthermore, by using the orthogonal decomposition method and the completing square method, the feedback representation of the optimal Nash equilibrium is derived for the first time. Finally, the effectiveness of our results is verified by a numerical example.
\end{abstract}

\begin{IEEEkeywords}
Discrete-time stochastic nonzero-sum game, inconsistent information structure, open-loop and closed-loop Nash equilibrium, maximum principle.
\end{IEEEkeywords}

\def\s{\small}
\def\n{\normalsize}
\def\+{\!+\!}
\def\-{\!-\!}
\def\={\!=\!}
\def\E{\mathbb{E}}
\def\P{\mathbb{P}}
\def\Tr{\mathrm{Tr}}
\def\b{\color{blue}}
\def\r{\color{red}}

\section{Introduction}

As is well known, due to the wide applications in industry, economics, management \cite{djl2000,cl2002}, the differential games have received many scholars' interest since 1950s, and abundant research results have been obtained, see \cite{i1965,ek1972,i1999,zp2022}. As a special case of differential games, the linear quadratic (LQ) non-cooperative stochastic game is a hot research topic, in view of its rigorous solution and the potential application background. For the LQ non-cooperative stochastic game problem, the system dynamics is described with a linear stochastic differential/difference equation, and the quadratic utility functions are to be minimized. It is worth mentioning that the open-loop and closed-loop Nash equilibrium were proposed in \cite{sy2019} for the continuous-time stochastic LQ two-person nonzero-sum differential game, it was shown that the existence of the open-loop Nash equilibrium is characterized by the forward-backward stochastic differential equations, and the closed-loop Nash equilibrium is characterized by the Riccati equations. Besides, references \cite{sjz2012a,sjz2012b} studied the discrete-time LQ non-cooperative stochastic game problem, and the Nash equilibrium was derived. Please refer to \cite{r2007,h1999,sy2019,slz2011,sjz2012a,sjz2012b,z2018,szl2021,cyl2022} and the cited references therein for the recent study on the LQ non-cooperative stochastic game.

It is noted that the previous works \cite{h1999,sy2019,slz2011,sjz2012a,sjz2012b,z2018,szl2021,cyl2022} on LQ stochastic game mainly focused on the case of the information structure being consistent, i.e., the two players share the same information sets. While, the information inconsistent case remains less investigated, i.e., the information sets obtained by the two players are different. As pointed out in \cite{nglb2014}, the asymmetry of information among the controllers makes it difficult to compute or characterize Nash equilibria. In fact, due to the inconsistent of the information structure, the two players are coupled with each other and finding the Nash equilibrium strategy becomes hard. Therefore, the study of the LQ non-zero stochastic game is challenging from the theoretical aspect.

In this paper,  a special case of LQ stochastic two-person nonzero-sum differential games with inconsistent information shall be studied, which is called the local and remote stochastic nonzero-sum game (LRSNG) problem. The detailed description of the networked system under consideration is shown in Figure \ref{Figure1}. It can be seen that the two players (the local player and the remote player) are included, and the uplink channel from the local player to the remote player is unreliable, while the downlink channel from the remote player to the local player is perfect. In other words, the precise state information can only be perfectly observed by the local player, and the remote player can just receive the disturbed state information delivered from the local player due to the unreliable uplink channel. Hence, the information sets accessed by the local player and the remote player are not the same, which is called inconsistent information structure. The objective of the two players is that each player should minimize its own quadratic cost function. Specifically, the two players are non-cooperative, therefore, the above problem is actually a LQ two-person nonzero-sum differential game with inconsistent information (LRSNG problem), and the goal of this paper is to find the Nash equilibrium associated with the LRSNG problem.

Actually, the corresponding local and remote control problem was originally studied in \cite{ott2016}, in which two controllers (the local controller and the remote controller) were cooperative, and the goal was to derive the jointly control law to optimize one common cost function. The optimal local and remote control problem is a decentralized control problem with asymmetric information structure, since the local controller and the remote controller can access different information sets, and the recent research results on the local and remote control can be found in \cite{qxz2020,lx2018,lqz2021,aon2019,tyw2022}. For instance, \cite{lx2018} investigated the optimal local and remote control problem, and a necessary and sufficient condition for the finite horizon optimal control problem was given by the use of maximum principle. In \cite{aon2019}, the local and remote control problem with multiple subsystems was studied, by using the common information approach, the optimal control strategy of the controller was obtained.
\begin{figure}[htbp]
  \centering
  \includegraphics[width=0.38\textwidth]{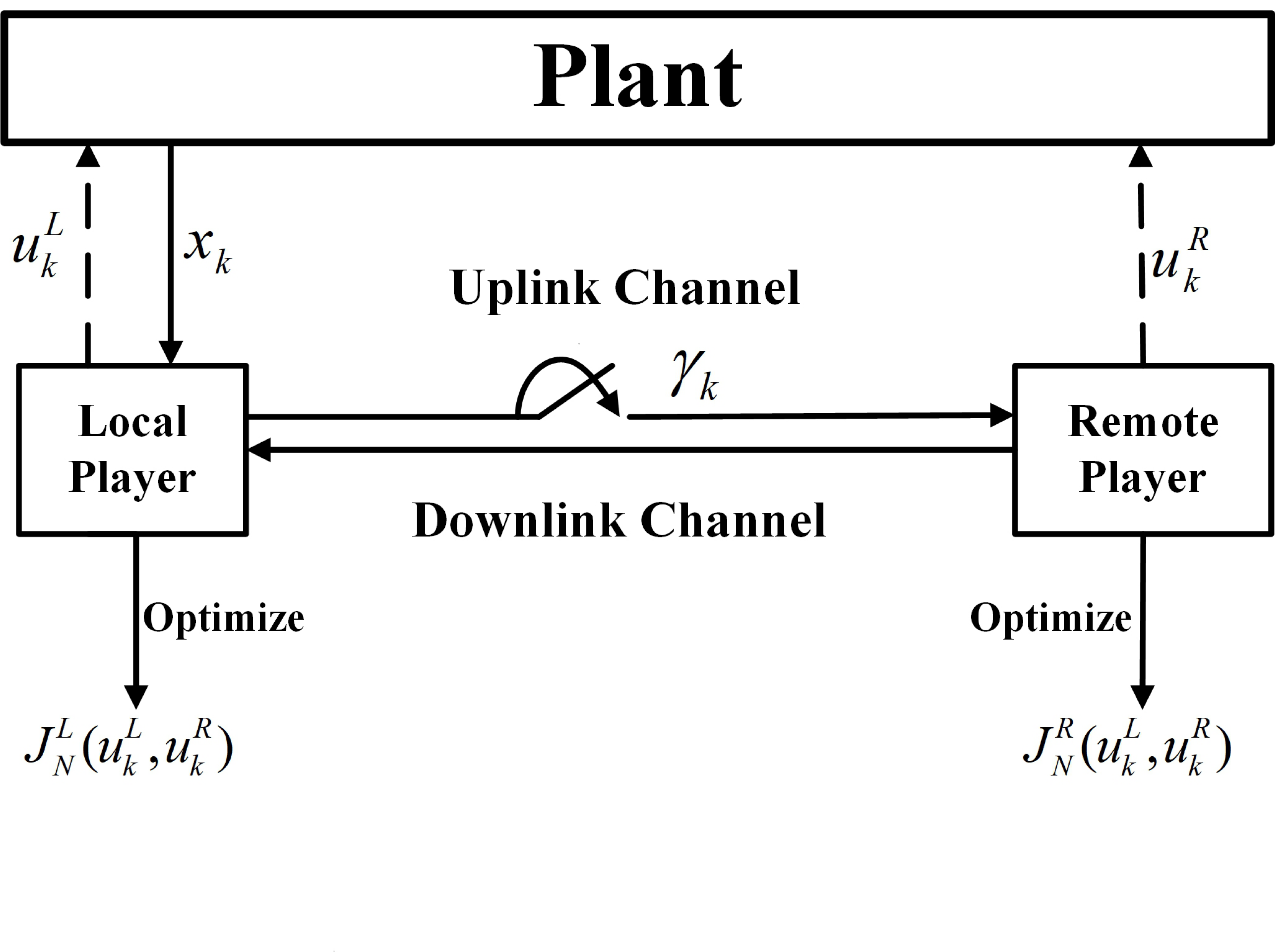}
  \caption{ Description of the LRSNG problem with inconsistent information structure.}
  \label{Figure1}
\end{figure}

While, the LRSNG problem studied in this paper is essentially different with the previous works \cite{qxz2020,lx2018,lqz2021,aon2019,tyw2022} on the optimal local and remote control in the following aspects: Firstly, the two players for the LRSNG problem need to be designed respectively to optimize their own cost function, i.e., they are non-cooperative. While, the two controllers for the local and remote control problem are cooperative, which is derived to minimize a common cost function. Secondly, although the local and remote control problem has been well studied, the LRSNG problem has not been solved before.

It is stressed that the study of the LRSNG problem is significant both from the theoretical and application aspects. From the theoretical point of view, as pointed out earlier, the LRSNG problem has not been solved before in view of the challenges from the inconsistent information structure. Furthermore, the LRSNG can be potentially applied in unmanned systems, manufacturing systems and autonomous vehicles, smart grid, remote surgery, etc., see \cite{lx2018,aon2019,hv2000,gc2010,hnx2007,lczsm2014}, and the cited references therein.

As analyzed above, the inconsistent information structure makes that the two players are coupled with each other, hence finding the Nash equilibrium for LRSNG problem becomes difficult. To overcome this challenge, the maximum principle and orthogonal decomposition approach are adopted in this paper to solve the LRSNG problem, and the open-loop and closed-loop Nash equilibrium are derived. Firstly, by the use of the convex variational method, the Pontryagin maximum principle is derived. Moreover, the necessary and sufficient conditions for the existence of the open-loop optimal Nash equilibrium for the LRSNG problem are derived, which are based on the solution of FBSDEs from the maximum principle. Consequently, in order to find a feedback explicit Nash equilibrium strategy, the orthogonal decomposition method and the completing square approach are applied, then the closed-loop optimal Nash equilibrium strategy for the LRSNG problem is derived for the first time, which is based on the solution to modified coupled Riccati equations. Finally, a numerical example is given to illustrate the main results of this paper.

In this paper, a special class of LQ non-cooperative stochastic game with inconsistent information (i.e., LRSNG) is firstly solved. The contributions of this paper are twofold. On the one hand, the open-loop Nash equilibrium for LRSNG problem is solved, the necessary and sufficient solvability conditions are derived. On the other hand, the closed-loop Nash equilibrium for LRSNG is developed, and the optimal feedback Nash equilibrium strategy is shown to rely on the solution to coupled Riccati equations.

The remainder of the paper is organized as follows. The LRSNG Problem is formulated in Section II. Section III solves the open-loop Nash equilibrium, and the solvability conditions are investigated. Section IV is devoted to solving the closed-loop Nash equilibrium for LRSNG problem. In Section V, the effectiveness of the main results is illustrated by numerical examples. The paper is concluded in Section VI.

For convenience, we will use the following notations throughout the paper. $\mathbb{R}^n$ signifies the $n$-dimensional Euclidean space. $I_n$ presents the unit matrix of $n\times n$ dimension. $A^\mathrm{T}$ denotes the transpose of the matrix $A$. $\mathcal F(X)$ denotes the $\sigma$-algebra generated by the random variable $X$. $A$ $\geq 0$ $(>0)$ means that $A$ is a positive semi-definite (positive definite) matrix. $Tr(A)$ represents the trace of matrix A. $\mathbb{E}[X]$ is the mathematical expectation, and $\mathbb{E}[X|\mathcal F_k]$ means the conditional expectation with respect to $\mathcal F_k$. The superscripts $L$, $R$ denote the local player and the remote player, respectively. $\inf$ means the infimum or the greatest lower bound of a set.
\section{Problem Formulation}

Throughout this paper, the following system dynamics shall be considered:
\begin{align}\label{ss1}
x_{k + 1} = Ax_k + {B^L}u_k^L + B^Ru_k^R + w_k,
\end{align}
where $x_k \in \mathbb{R}^n$ is the state, $u_k^L\in \mathbb{R}^{m_1}$ and $u_k^R\in \mathbb{R}^{m_2}$ are the local player and the remote player, respectively. $A$, $B^L$, $B^R$ are the constant system matrices with appropriate dimensions. $w_k$ is system noise with zero mean
and covariance $\Sigma_{w}$, taking values in $\mathbb{R}^n$. The initial value of state is $x_0$ with mean $\mu$ and covariance $\Sigma_{x_0}$, taking values in $\mathbb{R}^n$. $x_0$ and $w_k$, independent of each other, are Gaussian random variables.

As illustrated in Figure \ref{Figure1}, the uplink channel from the local player to the remote player is unreliable, while the downlink channel from the remote player to the local player is perfect. Thus, the information sets accessed by the local player and the remote player are given as follows, respectively:
\begin{align}\label{is1}
  \mathcal{F}_k^R&=\sigma \left\{\gamma_0x_0, \cdots, \gamma_kx_k \right\},\notag\\
  \mathcal{F}_k^L&=\sigma\left\{x_0, w_0, \cdots, w_{k-1},\gamma_0x_0, \cdots, \gamma_kx_k \right\},
\end{align}
in which $\gamma _k$ is an independent identically distributed Bernoulli random variable describing the state information transmitted through unreliable communication channel, i.e.,
\begin{equation}\label{uk3}
\gamma_k=\left\{ \begin{array}{ll}
  0, ~~\text{with probability}~~1-p,\\
 1, ~~\text{with probability}~~p.
\end{array} \right.
\end{equation}
In the above, $\gamma _k=1$ denotes that the state can be successfully accessed by the remote player, while $\gamma _k=0$
means the dropout of the state information from the local player to the remote player.

For the purpose of simplicity, the following notations are given:

\begin{align}\label{cs1}
 \mathcal{U}^L_N =\{u_0^L,\cdots,&u_N^L|u_k^L \in \mathbb{R}^{m_1}, u_k^L~is~\mathcal{F}_{k}^L-adapted,\notag\\
 &and~\sum\limits_{k = 0}^N \mathbb{E}[(u_k^L)^Tu_k^L]<+\infty \},\notag\\
 \mathcal{U}^R_N =\{u_0^R,\cdots,&u_N^R|u_k^R \in \mathbb{R}^{m_2}, u_k^R~is~\mathcal{F}_{k}^R-adapted,\notag\\
 &and~\sum\limits_{k = 0}^N \mathbb{E}[(u_k^R)^Tu_k^R]<+\infty\}.
\end{align}

The quadratic cost functions associated with system \eqref{ss1} are given by
\begin{small}
\begin{align}
J_N^L({u_k^L},{u_k^R})&= \sum\limits_{k = 0}^N \mathbb{E}[x_k^T{Q^L}{x_k}+(u_k^L)^T{S^L}u_k^L+(u_k^R)^T{M^L}{u_k^R}]\notag\\
&+\mathbb{E}[ x_{N + 1}^T{P^L_{N + 1}}{x_{N + 1}}]\label{cf1},\\
J_N^R({u_k^L},{u_k^R})&= \sum\limits_{k = 0}^N \mathbb{E}[x_k^T{Q^R}{x_k}+(u_k^L)^T{S^R}{u_k^L}+(u_k^R)^T{M^R}{u_k^R}]\notag\\
&+\mathbb{E}[ x_{N + 1}^T{P^R_{N + 1}}{x_{N + 1}}]\label{cf2},
\end{align}
\end{small}
where $Q^L$, $Q^R$, $S^L$, $M^L$, $S^R$, $M^R$, $P_{N+1}^L$, $P_{N+1}^R$ are given symmetric weighting matrices with compatible dimensions.

\begin{remark}\label{rm1}
It is stressed that the information sets $\mathcal{F}_{k}^L$ and $\mathcal{F}_k^R$ available to $u_k^L$ and $u_k^R$ are different, which is different from the consistent information structure case studied in previous works on LQ stochastic games \cite{slz2011,sjz2012a, sjz2012b,sy2019}. In fact, for $k=0,\cdots, N$, we have $\mathcal{F}_k^R \subseteq \mathcal{F}_{k}^L$, and the inconsistent information structure property would bring essential difficulties in solving the LQ stochastic two-person nonzero-sum game.
\end{remark}

Then, the open-loop and closed-loop LRSNG problems are stated as follows:

\textbf{Problem LRSNG.} For system \eqref{ss1} and cost functions \eqref{cf1}-\eqref{cf2}, find $u_k^L\in\mathcal{U}_N^L$ and $u_k^R\in\mathcal{U}_N^R$ to minimize $J_N^L$ and $J_N^R$, respectively.

\section{Open-loop Nash equilibrium}

In this section, we shall discuss the open-loop Nash equilibrium for Problem LRSNG in terms of FBSDEs, and the methods used are the convex variational principle and the maximum principle.

In the first place, the definition of the open-loop Nash equilibrium will be introduced.

\begin{definition}\label{def1}
 A pair $(u_k^{L,*},u_k^{R,*})$ $\in$ $\mathcal{U}_N^L \times \mathcal{U}_N^R$ is called an open-loop Nash equilibrium of Problem LRSNG if
\begin{align}
J_N^L\left(u_k^{L,*},u_k^{R,*} \right) \le J_N^L\left(u_k^{L},u_k^{R,*} \right),\forall u_k^L\in \mathcal{U}_N^L,\notag\\
J_N^R\left( u_k^{L,*},u_k^{R,*} \right) \le J_N^R\left(u_k^{L,*},u_k^{R} \right), \forall u_k^R\in \mathcal{U}_N^R.\label{pi1}
\end{align}
\end{definition}

Before stating the main results of this section, the following two lemmas will be given, which serve as preliminaries.
\begin{lemma}\label{lem1}
If we denote $(u_{k}^{L,*},u_{k}^{R,*})$ as the open-loop Nash equilibrium, then set $u_{k}^{L,\varepsilon}=u_{k}^{L,*}+\varepsilon \delta u_k^{L}$, $\delta u_k^{L}\in \mathcal{U}_N^L$, $\varepsilon\in\mathbb{R}$, ${z_{k}} = \frac{{x_{k}^\varepsilon  - {x_{k}}}}{\varepsilon }$,
and denote ${x_{k}^\varepsilon}$ and ${J_N^L}(u_k^{L,\varepsilon},u_k^{R,*})$ as the corresponding state and cost function associated with $u_{k}^{L,\varepsilon}$, $k =0, \cdots N$, then there holds
\begin{align}\label{lbf1}
&J_N^{L}(u_k^{L,\varepsilon},u_{k}^{R,*}) - {J_N^{L}(u_k^{L,*},u_{k}^{R,*})}\\
=&{\varepsilon ^2}\delta J_N^{L}(\delta u_k^{L})+ 2\varepsilon \sum\limits_{k = 0}^N \mathbb{E}[[(B^L)^T\theta _k^L + {S^L}u_k^{L,*}]^T\delta u_k^{L}],\notag
\end{align}
where $\delta J_N^{L}(\delta u_k^{L})$ is given by
\begin{align}\label{lbf2}
&\delta J_N^{L}(\delta u_k^{L})=\sum\limits_{k = 0}^N \mathbb{E}[z_k^T{Q^L}{z_k} + (\delta
u_k^{L})^T{S^L}\delta {u_k^{L}}]\notag\\
&+ \mathbb{E}[z_{N + 1}^TP_{N + 1}^L{z_{N + 1}}].
\end{align}
In the above, the costate ${\theta}_k^L (k = 0, \cdots, N)$ satisfies the following backward stochastic difference equation
\begin{equation}\label{lbf4}
\theta _{k-1}^{L}=Q^Lx_k+ \mathbb{E}\left[ A^T\theta _{k}^{L}|\mathcal{F}_k^L \right],\theta_N^L=P_{N+1}^{L}x_{N+1}.
\end{equation}
\end{lemma}
\begin{proof}
Using the notations introduced above, it can be derived that $z_k$ satisfies
\begin{align}\label{lbf3}
z_{k+1}=Az_k+B^L\delta u_k^{L},
\end{align}
with initial condition $z_0=0$.

Consequently, the variation of the cost function can be calculated as follows.
\begin{align*}
&{J_N^L}(u_k^{L,\varepsilon},u_k^{R,*}) - {J_N^L}(u_k^{L,*},u_k^{R,*})\\
&= \sum\limits_{k = 0}^N \mathbb{E}[{({x_k} + \varepsilon {z_k})}^T{Q^L}({x_k} + \varepsilon {z_k})\\
&+ {({u_k^{L,*}} + \varepsilon \delta {u_k^{L}})^T}{S^L}(u_k^{L,*} + \varepsilon \delta u_k^{L})\\
&+ (u_k^{R,*})^{T}{M^L}{u_k^{R,*}}]+ \mathbb{E}[({x_{N + 1}} + \varepsilon{z_{N + 1}})^T{P^L_{N + 1}}\\
&\times({x_{N + 1}}+ \varepsilon{z_{N + 1}})]- \mathbb{E}[x_{N + 1}^T{P^L_{N + 1}}{x_{N + 1}}]\\
&- \sum\limits_{k = 0}^N \mathbb{E}[x_k^T{Q^L}{x_k} + (u_k^{L,*})^TS^Lu_k^{L,*}+(u_k^{R,*})^{T}{M^L}{u_k^{R,*}}]\\
&= 2\varepsilon \mathbb{E}[\sum\limits_{k = 0}^N[x_k^TQ^Lz_k + (u_k^{L,*})^TS^L\delta u_k^{L}]\\
&+x_{N + 1}^TP^L_{N +1}z_{N + 1}] \\
&+ {\varepsilon ^2}\mathbb{E}[\sum\limits_{k = 0}^N[z_k^T{Q^L}{z_k} + (\delta u_k^{L})^T{S^L}\delta
u_k^{L}]\\
&+z_{N + 1}^TP_{N + 1}^L{z_{N + 1}}].
\end{align*}
Furthermore, from \eqref{lbf4} and \eqref{lbf3}, we have
\begin{align*}
&\mathbb{E}[\sum\limits_{k = 0}^N[x_k^TQ^Lz_k + (u_k^{L,*})^TS^L\delta u_k^{L}]\\
&+x_{N + 1}^TP^L_{N +1}z_{N + 1}] \\
&=\mathbb{E}[\sum\limits_{k = 0}^N[\theta _{k - 1}^L- E[A^T\theta
_k^L|\mathcal{F}_k^L]]^Tz_k \\
&+\sum\limits_{k = 0}^N(u_k^{L,*})^TS^L\delta u_k^{L}+(\theta _N^L)^Tz_{N + 1}] \\
&=\mathbb{E}[\sum\limits_{k = 0}^N[(B^L)^T\theta _k^L+S^Lu_k^{L,*}]^T\delta u_k^{L}].
\end{align*}
The proof is complete.
\end{proof}

Similar to Lemma \ref{lem1} and its proof, the following lemma can be given without proof.
\begin{lemma}\label{lem2}
For the open-loop Nash equilibrium $(u_{k}^{L,*},u_{k}^{R,*})$,
choose $\eta \in \mathbb{R}$, and for $k = 0, \cdots, N$, let $u_k^{R,\eta}= u_k^{R,*}+\eta\Delta u_k^{R}$, where $\Delta u_k^{R}\in \mathcal{U}_N^R$, $y_{k} = \frac{{x_{k}^\eta -{x_{k}}}}{\eta}$. Let ${x_{k}^\eta}$, $J_N^R(u_k^{L,*},u_k^{R,\eta})$ be the state and cost function associated with
$u_k^{R,\eta}$, $k =0, \cdots, N$, respectively. Then, we have
\begin{align}\label{rbf1}
&J_N^{R}(u_k^{L,*},u_k^{R,\eta}) - J_N^{R}(u_k^{L,*},u_k^{R,*})\notag\\
&={\eta ^2}\Delta J_N^{R}(\Delta u_k^{R})+2\eta \sum\limits_{k =0}^N \mathbb{E}[[{(B^R)^T}\theta _k^R + {M^R}u_k^{R,*}]^T\Delta u_k^{R}],
\end{align}
where $\Delta J_N^{R}(\Delta u_k^{R})$ is given by
\begin{align}\label{rbf2}
\Delta J_N^{R}(\Delta u_k^{R})=&\sum\limits_{k = 0}^N \mathbb{E}[y_k^T{Q^R}{y_k} + {{(\Delta u_k^{R})}^T}M^R\Delta u_k^{R}]\notag\\
&+ \mathbb{E}[y_{N + 1}^TP_{N + 1}^R{y_{N + 1}}].
\end{align}
And the costate $\theta_k^R(k =0, \cdots, N)$ satisfies
\begin{align}\label{rbf4}
\theta _{k - 1}^R = {Q^R}{x_k} + \mathbb{E}[{A^T}{\theta _k^R}|{\mathcal F_k^L}],
\end{align}
with terminal condition $\theta_{N}^R=P_{N+1}^Rx_{N+1}$.
\end{lemma}

Using the results derived in Lemmas \ref{lem1}-\ref{lem2}, the main results of this section will be presented, and the necessary and sufficient conditions for the open-loop Nash equilibrium of Problem LRSNG will be derived.

\begin{theorem}\label{th-01}
  For system \eqref{ss1} and cost functions \eqref{cf1} and \eqref{cf2}, the open-loop Nash equilibrium $(u_k^{L,*}, u_k^{R,*})$ for Problem LRSNG is unique if and only if the following two conditions are satisfied

  1) The convexity condition holds:
  \begin{align}\label{cc}
    \inf \delta J_N^L(\delta u_k^{L})\geq 0,~~\text{and}~~
    \inf \Delta J_N^R(\Delta u_k^{R})\geq 0,
  \end{align}
   in which $\delta J_N^L(\delta u_k^{L})$ and $\Delta J_N^R(\Delta u_k^{R})$ are given by \eqref{lbf2} and \eqref{rbf2}, respectively.

  2) The stationary conditions can be uniquely solved:
  \begin{align}
  0&=S^Lu_k^{L,*}+\mathbb{E}[(B^L)^T\theta _k^L|\mathcal{F}_k^L],\label{ec1}\\
  0&=M^Ru_k^{R,*}+ \mathbb{E}[(B^R)^T\theta _k^R|\mathcal{F}_k^R],\label{ec2}
  \end{align}
  where $\theta_k^L$, $\theta _k^R$ satisfy \eqref{lbf4} and \eqref{rbf4}, respectively.
\end{theorem}
\begin{proof}
`Necessity': Suppose the open-loop Nash equilibrium $(u_k^{L,*},u_k^{R,*})$ for Problem LRSNG is unique, we will show the two conditions 1)-2) are satisfied. In fact, for the open-loop Nash equilibrium $(u_k^{L,*},u_k^{R,*})$, from Lemmas \ref{lem1}-\ref{lem2}, we know that for arbitrary $\delta u_k^{L}\in \mathcal{U}_N^L$, $\varepsilon\in \mathbb{R}$, and arbitrary $\Delta u_k^{R}\in \mathcal{U}_N^R$, $\eta \in \mathbb{R}$,
there holds
\begin{align}\label{diff1}
&{J_N^L}(u_k^{L,\varepsilon},u_k^{R,*}) - {J_N^L}(u_k^{L,*},u_k^{R,*})\notag\\
= &2\varepsilon \sum\limits_{k = 0}^N \mathbb{E}[[(B^L)^T\theta _k^L + S^Lu_k^{L,*}]^T\delta
u_k^{L}]+{\varepsilon ^2}\delta J_N^L(\delta u_k^{L}) \notag\\
\geq& 0,
\end{align}
and
\begin{align}\label{diff2}
&J_N^R(u_k^{L,*},u_k^{R,\eta}) - J_N^R(u_k^{L,*},u_k^{R,*})\notag\\
=&2\eta \sum\limits_{k = 0}^N \mathbb{E}[[(B^R)^T\theta _k^R + {M^R}u_k^{R,*}]^T\Delta u_k^{R}]+{\eta ^2}\Delta J_N^R(\Delta u_k^{R}) \notag\\
\geq& 0.
\end{align}

On the one hand, suppose the convexity condition \eqref{cc} is not true, then there exists some $\delta u_k^{L}$, $k=0, \cdots, N$ such that
${J_N^L}(u_k^{L,\varepsilon} ,u_k^{R,*}) - {J_N^L}(u_k^{L,*},u_k^{R,*})=-\infty$ with
$\varepsilon\rightarrow\infty$. For the same reason, there exists some $\Delta u_k^{R}$, $k=0, \cdots, N$ satisfying
$J_N^R(u_k^{L,*},u_k^{R,\eta} )- J_N^R(u_k^{L,*},u_k^{R,*})=-\infty$ with $\eta\rightarrow\infty$.
This contradicts 1).

On the other hand, if 2) is not satisfied, then we can assume that
  \begin{align}
    S^Lu_k^{L,*}+\mathbb{E}[(B^L)^T\theta _k^L|\mathcal{F}_k^L]&=\Theta_k^L \neq 0,\label{u11}\\
    M^Ru_k^{R,*}+\mathbb{E}[(B^R)^T\theta _k^R|\mathcal{F}_k^R]&=\Theta_k^R \neq 0.\label{u12}
  \end{align}
In this case, if we choose $\delta u_k^{L}=\Theta_k^L$ and $\Delta u_k^{R}=\Theta_k^R$, then
from \eqref{diff1} and \eqref{diff2} we have
  \begin{align*}
  &{J_N^L}(u_k^{L,\varepsilon},u_k^{R,*}) - {J_N^L}(u_k^{L,*},u_k^{R,*})\\
  &= 2\varepsilon\sum_{k=0}^{N} (\Theta_k^L)^T\Theta_k^L +\varepsilon^2\delta J_N^L(\delta
  u_k^{L}),\\
  &J_N^R(u_k^{L,*},u_k^{R,\eta} ) - J_N^R(u_k^{L,*},u_k^{R,*})\\
  &=2\eta\sum_{k=0}^{N} (\Theta_k^R)^T\Theta_k^R +\eta^2\Delta J_N^R(\Delta u_k^{R}).
  \end{align*}
Then, we can always find some $\varepsilon$ and $\eta$ $<0$ such that ${J_N^L}(u_k^{L,\varepsilon},u_k^{R,*}) -
{J_N^L}(u_k^{L,*},u_k^{R,*})$and $J_N^R(u_k^{L,*},u_k^{R,\eta} ) - J_N^R(u_k^{L,*},u_k^{R,*})$ $<0$,
which contradicts with \eqref{diff1}, \eqref{diff2}. Thus, $\Theta_k^L$, $\Theta_k^R=0$, i.e.,
\eqref{ec1}, \eqref{ec2} holds. This ends the necessity proof.

`Sufficiency': Suppose the two conditions 1)-2) hold, we need to prove that the open-loop Nash equilibrium $(u_k^{L,*}, u_k^{R,*})$ is unique.

Actually, it can be deduced from \eqref{lbf1} and \eqref{rbf1} that for any
$\varepsilon\in\mathbb{R}$, $\eta\in\mathbb{R}$ and $\delta u_k^{L}\in \mathcal{U}_N^L$, $\Delta u_k^{R}\in \mathcal{U}_N^R$, we have
\begin{align*}
{J_N^L}(u_k^{L,\varepsilon},u_k^{R,*}) - {J_N^L}(u_k^{L,*},u_k^{R,*})&=\varepsilon^2\delta
J_N^L(\delta u_k^{L})\geq 0,\\
J_N^R(u_k^{L,*},u_k^{R,\eta} ) - J_N^R(u_k^{L,*},u_k^{R,*})&={\eta ^2}\Delta J_N^R(\Delta
u_k^{R})\geq 0,
\end{align*}
which means that open-loop Nash equilibria of Problem LRSNG is uniquely solvable. The proof is complete.
\end{proof}

\begin{remark}\label{rm2}
By using the variational method, the maximum principle for Problem LRSNG is derived. Furthermore, Theorem \ref{th-01} provides the necessary and sufficient solvability conditions for the open-loop Nash equilibrium of Problem LRSNG with inconsistent information structure for the first time.
\end{remark}

\begin{remark}\label{rm3}
From Theorem \ref{th-01}, the forward and backward difference equations (FBSDEs) can be given as follows
\begin{align}\label{fbsde}
\left\{ \begin{array}{ll}
&x_{k + 1}= Ax_k + {B^L}u_k^L + B^Ru_k^R + w_k,\\
&\theta _{k-1}^{L}=Q^Lx_k+ \mathbb{E}\left[ A^T\theta_{k}^{L}|\mathcal{F}_k^L \right],\\
&\theta _{k - 1}^R= {Q^R}{x_k} + \mathbb{E}[{A^T}{\theta _k^R}|{\mathcal F_k^L}],\\
&0=S^Lu_k^{L,*}+\mathbb{E}[(B^L)^T\theta _k^L|\mathcal{F}_k^L],\\
&0=M^Ru_k^{R,*}+ \mathbb{E}[(B^R)^T\theta _k^R|\mathcal{F}_k^R],\\
&\theta_N^L=P_{N+1}^{L}x_{N+1},\theta_N^R=P_{N+1}^{R}x_{N+1}.
\end{array} \right.
\end{align}
Due to the different information structure caused by the unreliable uplink channel, the FBSDEs \eqref{fbsde} cannot be decoupled as the traditional consistent information case, see \cite{nglb2014,nb2012}. Hence the explicit feedback Nash equilibrium cannot be derived via solving FBSDEs \eqref{fbsde}, which is challenging.
\end{remark}

\section{Closed-loop Nash equilibrium}

In this section, the closed-loop Nash equilibrium for Problem LRSNG will be studied. As illustrated in Remark \ref{rm3}, in view of the existence of inconsistent information structure, the FBSDEs  cannot be decoupled, which indicates the explicit Nash equilibrium cannot be derived via decoupling FBSDEs \eqref{fbsde} of Theorem \ref{th-01}. Alternatively, our aim is to obtain a feedback explicit Nash equilibrium (the closed-loop Nash equilibrium) by the use of the orthogonal decomposition and completing square approaches.

\subsection{Preliminaries}
To begin with, some preliminary results shall be introduced on the orthogonal decomposition method in deriving a feedback explicit Nash equilibrium (the closed-loop Nash equilibrium).

Without loss of generality, the following standard assumption will be made, see \cite{eln2013,y2013,ls1995}.
\begin{assumption}\label{ass1}
The weighting matrices in \eqref{cf1}, \eqref{cf2} satisfy
$Q^L\ge 0$, $Q^R\ge 0$, $S^L>0$, $S^R>0$, $M^L>0$, $M^R>0$, and $P_{N+1}^L\ge 0$, $P_{N+1}^R\ge 0$.
\end{assumption}

For the sake of discussion, the following notations are introduced.
\begin{align}\label{ncl}
&{\Lambda^L}=
\begin{bmatrix} S^L &  \\
 &  M^L \end{bmatrix},
{\Lambda^R}=
\begin{bmatrix} S^R &  \\
 &  M^R \end{bmatrix},
U_k=\begin{bmatrix}
\hat{u}_k^L\\
	u_k^R\\
\end{bmatrix},\notag\\
&\mathcal {B}=\begin{bmatrix} B^L&B^R \end{bmatrix},
\hat{u}_k^L=\mathbb{E}[u_k^L|\mathcal F_k^R],
\tilde u_k^L=u_k^L-\hat {u}_k^L.
\end{align}

Consequently, due to the existence of the unreliable uplink channel from the local player to the remote player, based on the disturbed state information, an estimator should be derived, which will be presented in the following lemma.
\begin{lemma}\label{lem3}
  For system \eqref{ss1} and the disturbed state \eqref{is1}, in the sense of minimizing the error covariance, using the notations introduced in \eqref{ncl}, the optimal estimator $\hat x_{k|k}=\mathbb{E}[x_k|\gamma_0x_0,\cdots, \gamma_kx_k]$ can be calculated as
\begin{align}\label{kw6}
\hat{x}_{k|k}={\gamma _k}x_k+(1 - {\gamma _k})(A\hat x_{k-1|k-1}+\mathcal BU_{k-1}),
\end{align}
with initial condition $\hat x_{0|0}={\gamma _0}x_0+(1 - {\gamma _0})\mu$, and $\mu$ is the mean value of the initial state $x_0$. Moreover, the estimation error ${{\tilde x}_k} = {x_k} - {{\hat x}_{k|k}}$ satisfies
\begin{align}\label{kw7}
\tilde{x}_k= (1 - {\gamma _k})(A{{\tilde x}_{k - 1}}
+ {B^L}{{\tilde u}_{k - 1}^L} + {w_{k - 1}}),
\end{align}
with initial condition $\tilde{x}_0=(1-\gamma _0)(x_0-\mu)$.
\end{lemma}

\begin{proof}
The detailed proof can be found in \cite{qz2017a,qz2017b}, which is omitted here.
\end{proof}

In order to be consistent with the information structure introduced in \eqref{cs1}, the form of the feedback explicit Nash equilibrium of Problem LRSNG is assumed as follows, see  \cite{sjz2012a,sjz2012b,ott2016,aon2019}. Specifically, based on \eqref{ncl} and Lemma \ref{lem3}, we assume:
\begin{assumption}\label{ass2}
On the one hand, $U_k$ is the feedback of the optimal estimator $\hat{x}_{k|k}$, i.e., $U_k=\tilde{K}_k^{L}\hat{x}_{k|k}$. On the other hand, $\tilde u_k^{L}$ is the feedback of the optimal estimation error $\tilde x_k$, i.e., $\tilde u_k^{L}=\tilde{K}_k^R\tilde x_k$.
\end{assumption}

From Assumption \ref{ass2}, we know that the feedback explicit Nash equilibrium (the closed-loop Nash equilibrium) is of the following form:
\begin{align}\label{form}
  u_k^L & =[I_{m_1} ~ 0]\tilde{K}_k^{L}\hat{x}_{k|k}+\tilde{K}_k^R\tilde x_{k},~\text{and}~
  u_k^{R} =[0 ~ I_{m_2}]\tilde{K}_k^{L}\hat{x}_{k|k}.
\end{align}

In the following, we will introduce the definition of closed-loop Nash equilibrium.
\begin{definition}\label{def2}
Under Assumption \ref{ass2}, a pair $(K_k^{L},K_k^{R})$ is called the closed-loop Nash equilibrium of Problem LRSNG if for any $(\tilde{K}_k^{L},\tilde{K}_k^{R})\in \mathbb{R}^{(m_1+m_2) \times n}\times \mathbb{R}^{m_1 \times n}$, the following relationships hold:
\begin{align}
J_N^L&([I_{m_1} ~~ 0]K_k^{L}\hat{x}_{k|k}+{K}_k^{R}\tilde x_{k},[0 ~~ I_{m_2}]K_k^{L}\hat{x}_{k|k})\label{clne1}\\
&\leq J_N^L([I_{m_1} ~~ 0]\tilde{K}_k^{L}\hat{x}_{k|k}+{K}_k^{R}\tilde x_{k},[0 ~~ I_{m_2}]\tilde{K}_k^L\hat{x}_{k|k}),\notag\\
J_N^R&([I_{m_1} ~~ 0]K_k^{L}\hat{x}_{k|k}+{K}_k^{R}\tilde x_{k},[0 ~~ I_{m_2}]K_k^{L}\hat{x}_{k|k})\label{clne2}\\
&\leq J_N^R([I_{m_1} ~~ 0]K_k^{L}\hat{x}_{k|k}+\tilde{K}_k^R\tilde x_{k},[0 ~~ I_{m_2}]K_k^{L}\hat{x}_{k|k}).\notag
\end{align}
\end{definition}

From Assumption \ref{ass2} and Definition \ref{def2}, the following lemma on the orthogonal property can be directly derived.
\begin{lemma}\label{lem4}
Under Assumption \ref{ass2}, for arbitrary $(\tilde{K}_k^{L},\tilde{K}_k^{R})\in \mathbb{R}^{(m_1+m_2) \times n}\times \mathbb{R}^{m_1 \times n}$, $U_k$ and $\tilde u_k^{L}$ given in Assumption \ref{ass2} are orthogonal, i.e., $\mathbb{E}[U_k^TH\tilde u_k^{L}]=0$ for any constant matrix $H$ with compatible dimensions.
\end{lemma}
\begin{proof}
In fact, from Assumption \ref{ass2}, we know that $U_k=\tilde{K}_k^L\hat{x}_{k|k}$ and $\tilde u_k^{L}=\tilde{K}_k^R\tilde x_k$, hence
\begin{align}
\mathbb{E}[U_k^TH\tilde{u}_k^L]
=&\mathbb{E}[(\tilde{K}_k^L\hat{x}_{k|k})^TH\tilde{K}^R_k\tilde{x}_k]\notag\\
=&\mathbb{E}[\hat{x}_{k|k}^T(\tilde{K}_k^L)^TH\tilde{K}^R_k\tilde{x}_k]\notag\\
=&0,
\end{align}
in which the orthogonality of $\hat{x}_{k|k}$ and $\tilde{x}_k$ has been inserted.
\end{proof}
\begin{remark}
In fact, noting that $u_k^L$ is $\mathcal{F}_{k}^L$-adapted, $u_k^R$ is $\mathcal{F}_{k}^R$-adapted, and $\mathcal{F}_k^R \subseteq \mathcal{F}_{k}^L$. Based on this basic property, hence Assumption \ref{ass2} is given, the orthogonal of $U_k$ and $\tilde u_k^{L}$ is shown in Lemma \ref{lem4}. The above method is called the orthogonal decomposition method, which is important in deriving the closed-loop Nash equilibrium.
\end{remark}

Using the results of Lemma \ref{lem4}, the cost functions \eqref{cf1}-\eqref{cf2} can be equivalently rewritten as follows:
\begin{align}
J_N^L(u_k^L, u_k^R)\notag
&= \sum\limits_{k = 0}^N \mathbb{E}[x_k^T{Q^L}{x_k} + U_k^T{\Lambda^L}{U_k} + (\tilde u_k^L)^T{S^L}{\tilde
u_k^L}]\notag\\
&+\mathbb{E}[ x_{N + 1}^T{P_{N + 1}^L}{x_{N + 1}}],\label{kw2}\\
J_N^R({u_k^L},u_k^R)\notag
&= \sum\limits_{k = 0}^N \mathbb{E}[x_k^T{Q^R}{x_k} + U_k^T{\Lambda^R}{U_k} + (\tilde u_k^L)^T{S^R}{\tilde
u_k^L}]\notag\\
&+\mathbb{E}[ x_{N + 1}^T{P_{N + 1}^R}{x_{N + 1}}]\label{kw3}.
\end{align}
Besides, for simplicity, we can always rewrite system dynamics \eqref{ss1} as:
\begin{align}\label{kw1}
x_{k+1}=Ax_k+\mathcal BU_k+B^L\tilde{u}_k^L+w_k.
\end{align}

\subsection{The closed-loop Nash equilibrium}

Before presenting the main results of closed-loop Nash equilibrium, we will introduce the following coupled Riccati equations in the first place.
{\small\begin{align}\label{rere}
\left\{ \begin{array}{ll}
P_k^L&= {A^T} P_{k+1}^LA+ Q^L-(K_k^L)^T(\Lambda^L+\mathcal {B}^TP_{k+1}^L\mathcal {B})K_k^L,\\
P_k^R&={A^T} P_{k+1}^RA + Q^R-(K_k^R)^T[{S^R}+(B^L)^TP_{k+1}^RB^L]K_k^R\\
&+p[(A + B^LK_k^R)^T \Omega_{k+1}^R(A + B^LK_k^R)\\
&- {(A + B^LK_k^R)^T} P_{k+1}^R(A + B^LK_k^R)],\\
\Omega_k^L&=p[(A + B^LK_k^R)^TP_{k+1}^L(A + B^LK_k^R)]\\
&+(1-p)[(A + B^LK_k^R)^T\Omega_{k+1}^L(A + B^LK_k^R)]\\
&+{Q^L}+ (K_k^R)^T{S^L}K_k^R,\\
\Omega_k^R&={(A + \mathcal {B}K_k^L)^T} \Omega_{k+1}^R(A + \mathcal {B}K_k^L)\\
&+ {Q^R}+ (K_k^L)^T\Lambda^{R}K_k^L,\\
K_k^{L}&=-(\Lambda^L+\mathcal {B}^TP_{k+1}^L\mathcal B)^{-1}\mathcal {B}^TP_{k+1}^LA,\\
K_k^{R}&=-[S^R+(B^L)^TP_{k+1}^RB^L]^{-1}(B^L)^TP_{k+1}^RA,\\
&S^R+(B^L)^TP_{k+1}^RB^L>0,
\end{array} \right.
\end{align}}
with terminal conditions $\Omega_{N+1}^L=P_{N+1}^L$, $\Omega_{N+1}^R=P_{N+1}^R$ and $P_{N+1}^L, P_{N+1}^R$ are given in \eqref{cf1}-\eqref{cf2}, respectively.

\begin{theorem}\label{th-02}
Under Assumptions \ref{ass1} and \ref{ass2}, if the coupled Riccati equation \eqref{rere} is solvable, then the closed-loop Nash equilibrium of Problem LRSNG is unique, and the closed-loop optimal Nash equilibrium is derived as
\begin{align}
    u_k^{L,*} & =[I_{m_1} ~~ 0]K_k^{L}\hat{x}_{k|k}+{K}_k^{R}\tilde x_{k}\label{occ1},\\
    u_k^{R,*} & =[0 ~~ I_{m_2}]K_k^{L}\hat{x}_{k|k}\label{occ2},
\end{align}
where $\hat{x}_{k/k}$ and $\tilde x_{k}$ are the optimal estimator and estimation error given in Lemma \ref{lem3}, and the gain matrices
$K_k^L,K_k^R$ can be calculated via \eqref{rere}.

With the closed-loop Nash equilibrium \eqref{occ1}-\eqref{occ2}, the corresponding optimal cost functions can be respectively calculated as follows.
\begin{align}
  &J_N^L(u_k^{L,*}, u_k^{R,*})
= \mathbb{E}[\hat x_{0|0}^TP_0^L\hat x_{0|0}+\tilde x_0^T \Omega_0^L\tilde {x}_0]\notag\\
  &~~~~~+p\sum\limits_{k = 0}^N Tr(\Sigma _wP_{k+1}^L)
  +(1 - p)\sum\limits_{k = 0}^N Tr(\Sigma _w\Omega_{k+1}^L)\label{ocf1},\\
  &J_N^R(u_k^{L,*}, u_k^{R,*})
=\mathbb{E}[{{\hat x}_{0|0}}^T\Omega_0^R{\hat x}_{0|0} + \tilde {x}_0^TP_0^R\tilde {x}_0]\notag\\
  &~~~~~+p\sum\limits_{k = 0}^NTr(\Sigma _w\Omega_{k+1}^R)
  +(1 - p)\sum\limits_{k = 0}^NTr(\Sigma _wP_{k+1}^R)\label{ocf2}.
\end{align}
\end{theorem}


Before we give the proof of Theorem \ref{th-02}, we will show the following propositions, which will be useful in deriving the main results.

\begin{proposition}\label{prop1}
Under Assumption \ref{ass1}, $\Lambda^L+\mathcal {B}^TP_{k+1}^L\mathcal B$ is positive definite for $k=0, \cdots, N$.
\end{proposition}
\begin{proof}
The backward induction method will be adopted to show that $\Lambda^L+\mathcal {B}^TP_{k+1}^L\mathcal B>0$ for $k=0, \cdots, N$.

Actually, from Assumption \ref{ass1}, we know that $P_{N+1}^L \geq 0$, and $\Lambda^L = \begin{bmatrix} S^L &  \\
 &  M^L \end{bmatrix} > 0$. Then, $\Lambda^L+\mathcal {B}^TP_{N+1}^L\mathcal B >0$ can be derived, and \eqref{rere} is solvable for $k=N$, that is
\begin{align}\label{pr1}
P_N^L&=Q^L+A^TP_{N+1}^LA-(K_N^L)^T(\Lambda^L+\mathcal {B}^TP_{k+1}^L\mathcal {B})K_N^L\notag\\
&=Q^L+A^TP_{N+1}^LA+(K_N^L)^T\mathcal {B}^TP_{N+1}^LA\notag\\
&+A^TP_{N+1}^L\mathcal {B}K_N^L+(K_N^L)^T(\Lambda^L+\mathcal {B}^TP_{k+1}^L\mathcal {B})K_N^L\notag\\
&=Q^L+(K_N^L)^T\Lambda^LK_N^L\notag\\
&+(A+\mathcal {B}K_N^L)^TP_{N+1}^L(A+\mathcal {B}K_N^L)
\end{align}
where $K_N^L$ is given as in \eqref{rere} for $k=N$.

Notice that $\Lambda^L>0$, $Q^L\geq0$, $P_{N+1}^L \geq 0$, then $P_N^L\geq 0$ can be obtained from \eqref{pr1}.

By repeating the above procedures step by step backwardly, we can conclude that $P_k^L\geq0$, which indicates that $\Lambda^L+\mathcal {B}^TP_{k+1}^L\mathcal B>0$ for any $0\leq k\leq N$.
\end{proof}

From the definition of the closed-loop Nash equilibrium (Definition \ref{def2}), the following two propositions are to be shown.
\begin{proposition}\label{prop2}
  Under Assumptions \ref{ass1}-\ref{ass2}, for the closed-loop Nash equilibrium $(K_k^{L}, K_k^{R})$, which minimizes $J_N^L(u_k^L,u_k^R)$ (see Definition \ref{def2}), if $\tilde u_k^{L,*}=K_k^{R}\tilde{x}_{k}$ is given in advance, then $U_k^*$ can be calculated as
  \begin{align}\label{mathu}
    U_k^*=K_k^L\hat{x}_{k|k},
  \end{align}
  in which $K_k^L$ is given by
  \begin{align}\label{kkl}
    K_k^{L}&=-(\Lambda^L+\mathcal {B}^TP_{k+1}^L\mathcal {B})^{-1}\mathcal {B}^T{P}_{k+1}^LA,
  \end{align}
  with $P_k^L$ satisfying
  \begin{align}
    P_k^L&= {A^T} P_{k+1}^LA-(K_k^L)^T(\Lambda^L+\mathcal {B}^TP_{k+1}^L\mathcal {B})K_k^L\notag\\
    &+ {Q^L},~ P_{N+1}^L,\label{pkl}\\
    \Omega_k^L&=p[(A + B^LK_k^R)^TP_{k+1}^L(A + {B}^LK_k^R)]\notag\\
&+(1-p)[(A + B^LK_k^R)^T\Omega_{k+1}^L(A + {B}^LK_k^R)]\notag\\
&+{Q^L}+ (K_k^R)^TS^LK_k^R,~~\Omega_{N+1}^L=P_{N+1}^L.\label{okl}
  \end{align}
  In this case, the optimal $J_N^L(u_k^{L,*}, u_k^{R,*})$ is given by \eqref{ocf1}.
\end{proposition}
\begin{proof}
For the sake of discussion, we denote $V_N^L(\hat{x}_{k|k},\tilde{x}_k)$ as follows:
\begin{align}\label{vnlk}
  V_N^L(\hat{x}_{k|k},\tilde{x}_k)&=\mathbb{E}[\hat x_{k|k}^T P_k^L\hat {x}_{k|k}+\tilde x_k^T \Omega_k^L\tilde {x}_k],
\end{align}
where $P_k^L$, $\Omega_k^L$ satisfy the equations \eqref{pkl}-\eqref{okl}.

Noting that $K_k^{R}$ is given, i.e.,  $\tilde{u}_k^{L,*}=K_k^{R}\tilde{x}_{k}$, hence we have,
\begin{align}\label{lzhs1}
&V_N^L(\hat{x}_{k|k},\tilde{x}_k)- V_N^L(\hat{x}_{k+1|k+1},\tilde{x}_{k+1})\notag\\
&= \mathbb{E}[\hat x_{k|k}^TP_k^L{{\hat x}_{k|k}}+\tilde x_k^T \Omega_k^L\tilde {x}_k]\notag\\
&- \mathbb{E}[\hat x_{k + 1|k + 1}^T P_{k+1}^L{{\hat x}_{k + 1|k + 1}}]-\mathbb{E}[\tilde x_{k + 1}^T \Omega_{k+1}^L\tilde {x}_{k + 1}]\notag\\
&= \mathbb{E}[\hat x_{k|k}^TP_k^L\hat {x}_{k|k}+\tilde {x}_k^T \Omega_k^L\tilde {x}_k]\notag\\
&-\mathbb{E}[[\gamma _{k + 1}(A\hspace{-1mm}+\hspace{-1mm} B^LK_k^{R})\tilde {x}_k \hspace{-0.5mm}+\hspace{-0.5mm} \gamma _{k + 1}w_k\hspace{-1mm} +\hspace{-1mm} A\hat {x}_{k|k}
\hspace{-1mm}+\hspace{-1mm}\mathcal {B}U_k]^TP_{k+1}^L\notag\\
&\times [\gamma _{k + 1}(A\hspace{-1mm}+\hspace{-1mm} B^LK_k^{R})\tilde {x}_k \hspace{-0.5mm}+ \hspace{-0.5mm}\gamma _{k + 1}w_k\hspace{-1mm} +\hspace{-1mm} A\hat {x}_{k|k}
\hspace{-1mm}+\hspace{-1mm}\mathcal {B}U_k]]\notag\\
&- \mathbb{E}[[(1 - \gamma _{k + 1})(A + B^LK_k^{R})\tilde {x}_k +(1 - \gamma _{k + 1}) w_k]^T\Omega_{k+1}^L\notag\\
&\times[(1 - \gamma _{k + 1})(A + B^LK_k^{R})\tilde {x}_k +(1 - \gamma _{k + 1}) w_k]]\notag\\
&= \mathbb{E}[-(U_k - K_k^L\hat{x}_{k|k})^T(\Lambda^L + \mathcal {B}^TP_{k+1}^L\mathcal {B})\notag\\
&\times(U_k - K_k^L\hat{x}_{k|k})+ U_k^T\Lambda^LU_k]\notag\\
&+\mathbb{E}[\tilde x_k^T[\Omega_k^L \hspace{-0.5mm}- \hspace{-0.5mm}(1-p)(A\hspace{-0.5mm} +\hspace{-0.5mm} B^LK_k^{R})^T \Omega_{k+1}^L(A \hspace{-0.5mm}+\hspace{-0.5mm} B^LK_k^{R})\notag\\
&-p(A + B^LK_k^{R})^T P_{k+1}^L(A + B^LK_k^{R})]\tilde {x}_k]\notag\\
&+\mathbb{E}[\hat{x}_{k|k}^T [P_k^L-A^TP_{k+1}^LA\notag\\
&+(K_k^L)^T(\Lambda^L+\mathcal {B}^TP_{k+1}^L\mathcal {B})K_k^L]\hat{x}_{k|k}]\notag\\
&-(1-p)Tr(\Sigma _w\Omega_{k+1}^L)-pTr(\Sigma _wP_{k+1}^L).
\end{align}

Furthermore, from \eqref{kkl}-\eqref{okl}, it can be derived from \eqref{lzhs1} that
\begin{align}\label{lzhs2}
&V_N^L(\hat{x}_{k|k},\tilde{x}_k)- V_N^L(\hat{x}_{k+1|k+1},\tilde{x}_{k+1})\notag\\
&= \mathbb{E}[-(U_k - K_k^L\hat{x}_{k|k})^T(\Lambda^L + \mathcal {B}^TP_{k+1}^L\mathcal {B})\notag\\
&\times(U_k - K_k^L\hat{x}_{k|k})+ U_k^T\Lambda^LU_k]\notag\\
&+\mathbb{E}[\hat x_{k|k}^TQ^L\hat {x}_{k|k}]+\mathbb{E}[\tilde{x}_k^T[Q^L+(K_k^{R})^TS^LK_k^{R}]\tilde{x}_k]\notag\\
&-(1 - p)Tr(\Sigma _w\Omega_{k+1}^L)-pTr(\Sigma _wP_{k+1}^L).
\end{align}

Taking summation on both sides of \eqref{lzhs2} from $k=0$ to $k=N$, there holds
\begin{align}\label{lzhs3}
&V_N^L(\hat{x}_{0|0},\tilde{x}_0)- V_N^L(\hat{x}_{N+1|N+1},\tilde{x}_{N+1})\notag\\
&= \mathbb{E}[{\hat x_{0|0}}^TP_0^L\hat x_{0|0}+\tilde x_0^T\Omega_0^L\tilde {x}_0]\notag\\
&- \mathbb{E}[\hat x_{N+1|N+1}^TP_{N+1}^L\hat {x}_{N+1|N+1}+\tilde x_{N+1}^T\Omega_{N+1}^L\tilde {x}_{N + 1}]\notag\\
&= \sum\limits_{k = 0}^N  \mathbb{E}[-(U_k - K_k^L\hat{x}_{k|k})^T(\Lambda^L + \mathcal {B}^TP_{k+1}^L\mathcal {B})\notag\\
&\times(U_k - K_k^L\hat{x}_{k|k})+ U_k^T\Lambda^LU_k]\notag\\
&+ \mathbb{E}[\hat x_{k|k}^TQ^L\hat {x}_{k|k}] + \mathbb{E}[\tilde x_k^T[Q^L+(K_k^{R})^TS^LK_k^{R}]\tilde
{x}_k]\notag\\
&-(1 - p)\sum\limits_{k = 0}^N Tr(\Sigma _w \Omega_{k+1}^L)- p\sum\limits_{k = 0}^N Tr(\Sigma _wP_{k+1}^L).
\end{align}
Then, from \eqref{kw2} we have that
\begin{align}\label{ocf3}
&J_N^L(u_k^L, u_k^R)=\sum\limits_{k = 0}^N \mathbb{E}[\hat x_{k|k}^T{Q^L}{{\hat x}_{k|k}}+U_k^T\Lambda^LU_k\notag\\
& + \tilde x_k^T[Q^L+(K_k^{R})^TS^LK_k^{R}]\tilde {x}_k]+\mathbb{E}[x_{N + 1}^TP_{N+1}^Lx_{N + 1}]\notag\\
&= \sum\limits_{k = 0}^N \mathbb{E}[(U_k - K_k^L\hat{x}_{k|k})^T(\Lambda^L + \mathcal {B}^TP_{k+1}^L\mathcal {B})\notag\\
&\times(U_k - K_k^L\hat{x}_{k|k})]+ \mathbb{E}[\hat x_{0|0}^TP_0^L\hat x_{0|0}+\tilde x_0^T\Omega_0^L\tilde {x}_0]\notag\\
&+ (1 - p)\sum\limits_{k = 0}^N Tr(\Sigma _w \Omega_{k+1}^L)+ p\sum\limits_{k = 0}^NTr(\Sigma _wP_{k+1}^L).
\end{align}

As shown in Proposition \ref{prop1}, $\Lambda^L+\mathcal{B}^TP_{k+1}^L\mathcal{B}>0$, therefore, $J_N^L(u_k^L, u_k^R)$ can be minimized by \eqref{mathu} with $\tilde{u}_k^{L,*}=K_k^{R}\tilde{x}_{k}$ given in advance. This completes the proof.
\end{proof}

Similarly, the following proposition can be shown.
\begin{proposition}\label{prop3}
Suppose Assumptions \ref{ass1}-\ref{ass2} hold, and denote $(K_k^{L}, K_k^{R})$ as the closed-loop Nash equilibrium of optimizing $J_N^R(u_k^L,u_k^R)$, if we set $ U_k^*=K_k^L\hat{x}_{k|k}$ in advance, then $ \tilde{u}_k^{L,*}$ is given by
  \begin{align}\label{tildeu}
   \tilde{u}_k^{L,*}&=K_k^{R}\tilde{x}_{k}.
  \end{align}
In the above, $K_k^R$ satisfies
{\small\begin{align}\label{kkr}
\left\{ \begin{array}{ll}
K_k^{R}&\hspace{-2mm}=-[S^R+(B^L)^T{P}_{k+1}^RB^L]^{-1}(B^L)^T{P}_{k+1}^RA,\\
P_k^R&\hspace{-2mm}={A^T} P_{k+1}^RA + Q^R-(K_k^R)^T[S^R+(B^L)^TP_{k+1}^RB^L]K_k^R\\
&+p[(A + B^LK_k^R)^T \Omega_{k+1}^R(A + B^LK_k^R)\\
&- {(A + B^LK_k^R)^T} P_{k+1}^R(A + B^LK_k^R)],\\
\Omega_k^R&\hspace{-2mm}=(A + \mathcal {B}K_k^L)^T \Omega_{k+1}^R(A + \mathcal {B}K_k^L)\\
&+ Q^R+ (K_k^L)^T\Lambda^RK_k^L
\end{array} \right.
\end{align}}
Furthermore, $J_N^R(u_k^{L}, u_k^{R})$ can be minimized as \eqref{ocf2}.
\end{proposition}
\begin{proof}
  By following Proposition \ref{prop2} and its proof, we define
\begin{align*}
V_N^R(\hat{x}_{k|k},\tilde{x}_k) = \mathbb{E}[\hat x_{k|k}^T \Omega_k^R\hat {x}_{k|k} + \tilde x_k^TP_k^R\tilde {x}_k],
\end{align*}
in which $P_k^R$, $\Omega_k^R$ satisfy \eqref{kkr}.

Hence, it can be derived that
\begin{align}\label{rzhs1}
&V_N^R(\hat{x}_{k|k},\tilde{x}_k)-V_N^R(\hat{x}_{k+1|k+1},\tilde{x}_{k+1})\notag\\
&= \mathbb{E}[\hat x_{k|k}^T \Omega_k^R\hat {x}_{k|k} + \tilde x_k^TP_k^R\tilde {x}_k]\notag\\
&- \mathbb{E}[\hat x_{k + 1|k + 1}^T\Omega_{k+1}^R\hat x_{k + 1|k + 1}]- \mathbb{E}[\tilde x_{k +1}^TP_{k+1}^R\tilde {x}_{k + 1}]\notag\\
&= \mathbb{E}[\hat x_{k|k}^T \Omega_k^R{{\hat x}_{k|k}} + \tilde x_k^TP_k^R\tilde {x}_k]\notag\\
&- \mathbb{E}[[\gamma _{k + 1}(A\tilde {x}_k\hspace{-0.5mm} +\hspace{-0.5mm} B^L\tilde {u}_k^L\hspace{-0.5mm} +\hspace{-0.5mm} w_k)\hspace{-0.5mm} +\hspace{-0.5mm} (A+\mathcal BK_k^{L})\hat {x}_{k|k}]^T\Omega_{k+1}^R\notag\\
&\times [{\gamma _{k + 1}}(A\tilde {x}_k\hspace{-0.5mm} +\hspace{-0.5mm} B^L\tilde {u}_k^L\hspace{-0.5mm} +\hspace{-0.5mm} w_k)\hspace{-0.5mm} +\hspace{-0.5mm} (A+\mathcal BK_k^{L})\hat {x}_{k|k}]]\notag\\
&- \mathbb{E}[[(1 - \gamma _{k + 1})(A\tilde {x}_k + B^L\tilde {u}_k^L+
{w_k})]^TP_{k+1}^R\notag\\
&\times[(1 - \gamma _{k + 1})(A\tilde {x}_k + {B^L}\tilde {u}_k^L + w_k)]]\notag\\
&= \mathbb{E}[\hat x_{k|k}^T \Omega_k^R\hat {x}_{k|k} + \tilde x_k^TP_k^R\tilde {x}_k]-\mathbb{E}[\tilde{x}_k^TA^TP_{k+1}^RA\tilde{x}_k]\notag\\
&-\mathbb{E}[\hat{x}_{k|k}^T(A+\mathcal {B}K_k^{L})^T\Omega_{k+1}^R(A+\mathcal {B}K_k^{L})\hat{x}_{k|k}]\notag\\
&-\mathbb{E}[(\tilde{u}_k^L-K_k^R\tilde{x}_k)^T[S^R+(B^L)^TP_{k+1}^RB^L](\tilde{u}_k^L-K_k^R\tilde{x}_k)]\notag\\
&+\mathbb{E}[(\tilde{u}_k^L)^TS^R\tilde{u}_k^L]\hspace{-1mm}+\hspace{-1mm}\mathbb{E}[p(A\tilde{x}_k\hspace{-1mm}+\hspace{-1mm}B^L\tilde{u}_k^L)^TP_{k+1}^R(A\tilde{x}_k\hspace{-1mm}+\hspace{-1mm}B^L\tilde{u}_k^L)\notag\\
&-p(A\tilde{x}_k+B^L\tilde{u}_k^L)^T\Omega_{k+1}^R(A\tilde{x}_k+B^L\tilde{u}_k^L)]\notag\\
&+\mathbb{E}[\tilde{x}_k^T(K_k^R)^T[S^R+(B^L)^TP_{k+1}^RB^L]K_k^R\tilde{x}_k]\notag\\
&-pTr(\Sigma _w\Omega_{k+1}^R) - (1 - p)Tr(\Sigma _wP_{k+1}^R).
\end{align}
Next, by using \eqref{rere}, we obtain
\begin{align}\label{rzhs2}
&V_N^R(\hat{x}_{k|k},\tilde{x}_k)-V_N^R(\hat{x}_{k+1|k+1},\tilde{x}_{k+1})\notag\\
&=-\mathbb{E}[(\tilde{u}_k^L\hspace{-0.5mm}-\hspace{-0.5mm}K_k^R\tilde{x}_k)^T[S^R\hspace{-0.5mm}+\hspace{-0.5mm}(B^L)^TP_{k+1}^RB^L](\tilde{u}_k^L\hspace{-0.5mm}-\hspace{-0.5mm}K_k^R\tilde{x}_k)]\notag\\
&+\mathbb{E}[(\tilde{u}_k^L)^TS^R\tilde{u}_k^L]+ \mathbb{E}[\hat x_{k|k}^T[{Q^R} + (K_k^{L})^T\Lambda^RK_k^{L}]{\hat
x}_{k|k}]\notag\\
&+\mathbb{E}[\tilde {x}_k^TQ^R\tilde {x}_k]-pTr(\Sigma _w\Omega_{k+1}^R) - (1 - p)Tr(\Sigma _wP_{k+1}^R).
\end{align}

Then, by adding from $k=0$ to $k=N$ of \eqref{rzhs2}, there holds
\begin{align}\label{rzhs3}
&V_N^R(\hat{x}_{0|0},\tilde{x}_0)-V_N^R(\hat{x}_{N+1|N+1},\tilde{x}_{N+1})\notag\\
&= \mathbb{E}[\hat x_{0|0}^T\Omega_0^R\hat {x}_{0|0} + \tilde x_0^TP_0^R\tilde {x}_0]\notag\\
&- \mathbb{E}[\hat x_{N + 1|N + 1}^T\Omega_{N+1}^R{\hat x_{N + 1|N + 1}}- \tilde x_{N + 1}^TP_{N+1}^R\tilde {x}_{N + 1}]\notag\\
&= \sum\limits_{k = 0}^N-\mathbb{E}[(\tilde{u}_k^L-K_k^R\tilde{x}_k)^T[S^R+(B^L)^TP_{k+1}^RB^L]\notag\\
&\times(\tilde{u}_k^L-K_k^R\tilde{x}_k)]+\mathbb{E}[(\tilde{u}_k^L)^TS^R\tilde{u}_k^L+\tilde x_k^TQ^R\tilde {x}_k]\notag\\
&+ \mathbb{E}[\hat x_{k|k}^T[Q^R + (K_k^{L})^T\Lambda^RK_k^{L}]{{\hat
x}_{k|k}}]\notag\\
&-pTr(\Sigma _w\Omega_{k+1}^R) - (1 - p)Tr(\Sigma _wP_{k+1}^R).
\end{align}

Finally, from \eqref{kw3}, we have
\begin{align}\label{ocf4}
&J_N^R(u_k^{L}, u_k^{R})= \sum\limits_{k = 0}^N \mathbb{E}[\tilde x_k^TQ^R\tilde {x}_k+(\tilde{u}_k^L)^TS^R\tilde{u}_k^L] \notag\\
&+\mathbb{E}[ \hat x_{k|k}^T[Q^R + (K_k^R)^T{S^L}K_k^R]\hat {x}_{k|k}]+ \mathbb{E}[x_{N + 1}^TP_{N+1}^Rx_{N + 1}] \notag\\
&= \sum\limits_{k = 0}^N \mathbb{E}[(\tilde{u}_k^L-K_k^R\tilde{x}_k)^T[S^R+(B^L)^TP_{k+1}^RB^L]\notag\\
&\times(\tilde{u}_k^L-K_k^R\tilde{x}_k)]+ \mathbb{E}[\hat x_{0|0}^T\Omega_0^R\hat {x}_{0|0} + \tilde
x_0^TP_0^R\tilde {x}_0]\notag\\
&+p\sum\limits_{k = 0}^N Tr(\Sigma _w\Omega_{k+1}^R) + (1 - p)\sum\limits_{k = 0}^N Tr(\Sigma _wP_{k+1}^R).
\end{align}

Since $S^R+(B^L)^T{P}_{k+1}^RB^L>0$ given in \eqref{rere}, thus $J_N^R(u_k^{L}, u_k^{R})$ can be minimized by \eqref{tildeu}. The proof is complete.
\end{proof}

In the following, the proof of Theorem \ref{th-02} shall be given.\\
\begin{proof}
  \textbf{Proof of Theorem \ref{th-02}.} Combining Propositions \ref{prop2}-\ref{prop3}, we can conclude that if the coupled Riccati equations \eqref{rere} is solvable, then $(K_k^L, K_k^R)$ given in \eqref{mathu} and \eqref{tildeu} is the unique closed-loop Nash equilibrium, as given in \eqref{clne1}-\eqref{clne2} of Definition \ref{def2}). Therefore, from \eqref{ncl}, we know that the optimal action of $(u_k^{L,*},u_k^{R,*})$ can be given as \eqref{occ1}-\eqref{occ2}. Moreover, the optimal $J_N^L(u_k^{L,*}, u_k^{R,*}), J_N^R(u_k^{L,*}, u_k^{R,*})$ are given as \eqref{ocf1}-\eqref{ocf2}, this ends the proof.
\end{proof}

\begin{remark}
 In Theorem \ref{th-02}, the closed-loop Nash equilibrium of Problem LRSNG is derived in the feedback form, the obtained results are new to the best of our knowledge. For the closed-loop Nash equilibrium of Problem LRSNG, on the one hand, the local player $u_k^{L}$ should take the closed-loop optimal Nash equilibrium $u_k^{L,*}$ if and only if the remote player takes the closed-loop optimal feedback Nash equilibrium $u_k^{R,*}$. Meanwhile, the optimal cost function is given as the value of \eqref{ocf1}. On the other hand, the remote player $u_k^{R}$ should take the closed-loop optimal Nash equilibrium $u_k^{R,*}$ if and only if the local player takes the closed-loop optimal feedback Nash equilibrium $u_k^{L,*}$, and the optimal cost function \eqref{ocf2} is also derived.
 Besides, the calculation of closed-loop Nash equilibrium is based on the coupled Riccati equations.
\end{remark}

\section{Numerical Example}
\begin{figure}
  \centering
  \includegraphics[width=0.38\textwidth]{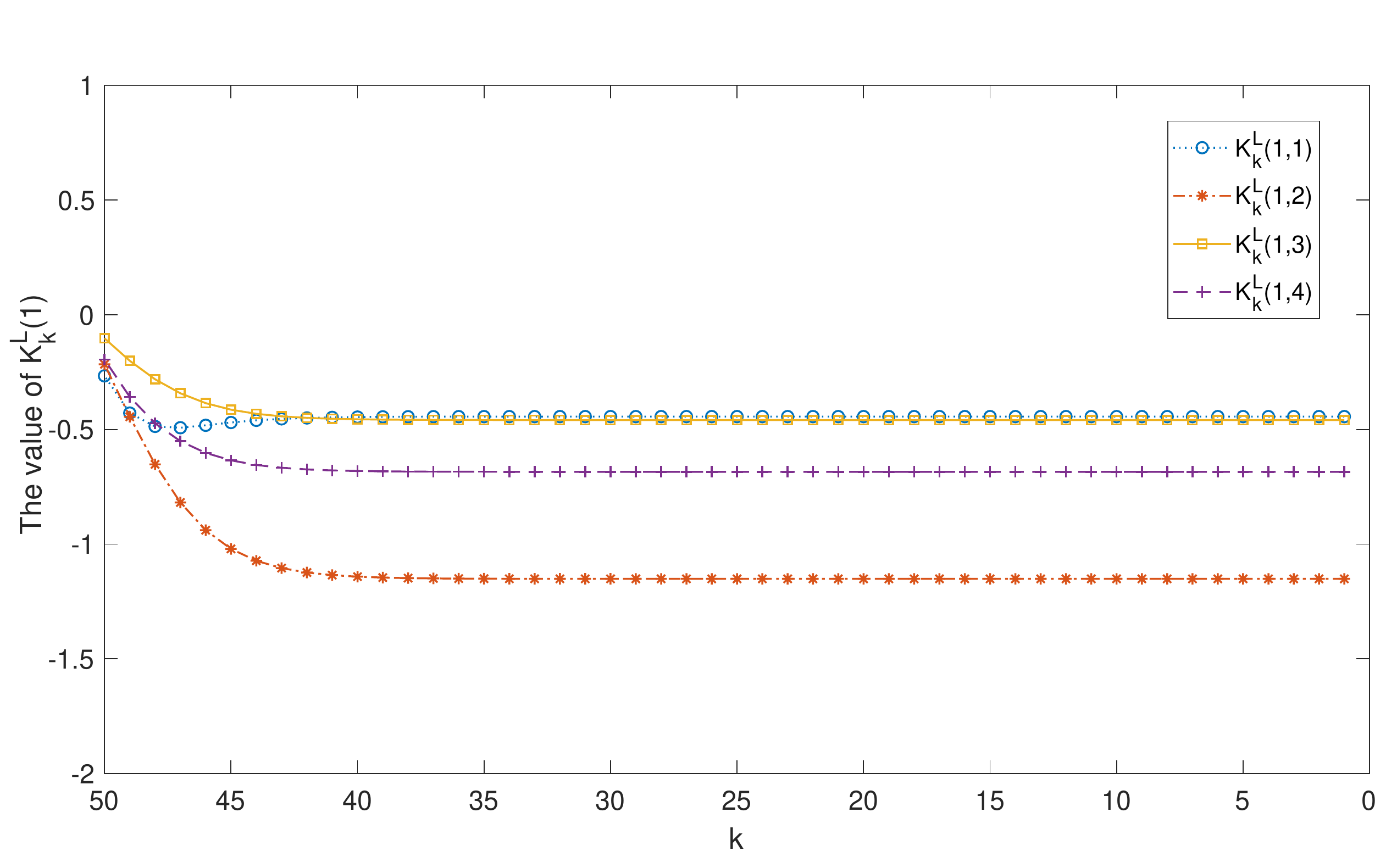}\\
  \caption{Closed-loop Nash equilibrium: $K_k^L(1,i), i=1,\cdots,4$, the first column value of $K_k^L$.}\label{Figure2}
\end{figure}
\begin{figure}
  \centering
  \includegraphics[width=0.38\textwidth]{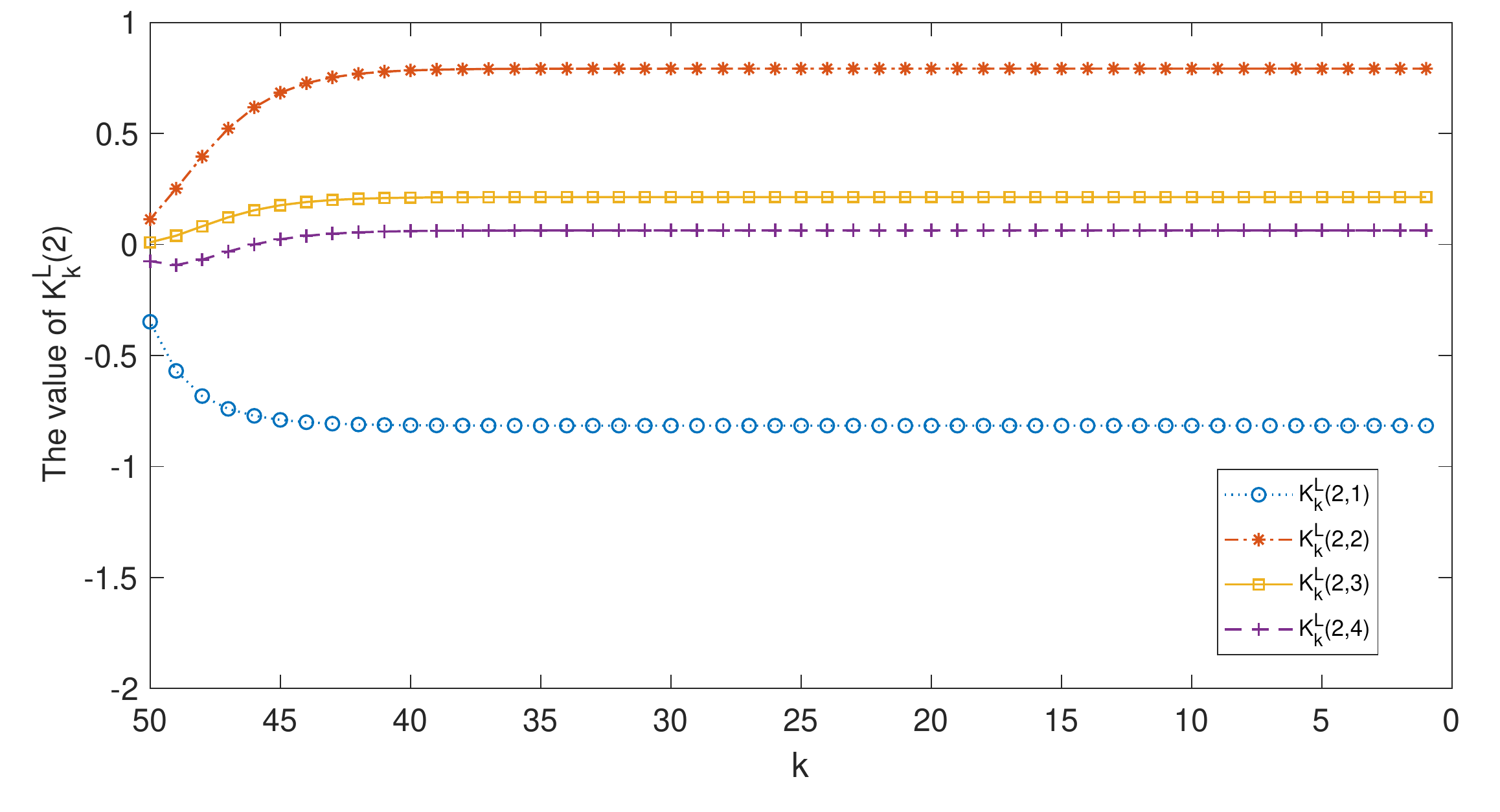}\\
  \caption{Closed-loop Nash equilibrium: $K_k^L(2,i), i=1,\cdots,4$, the second column value of $K_k^L$.}\label{Figure3}
\end{figure}
\begin{figure}
  \centering
  \includegraphics[width=0.38\textwidth]{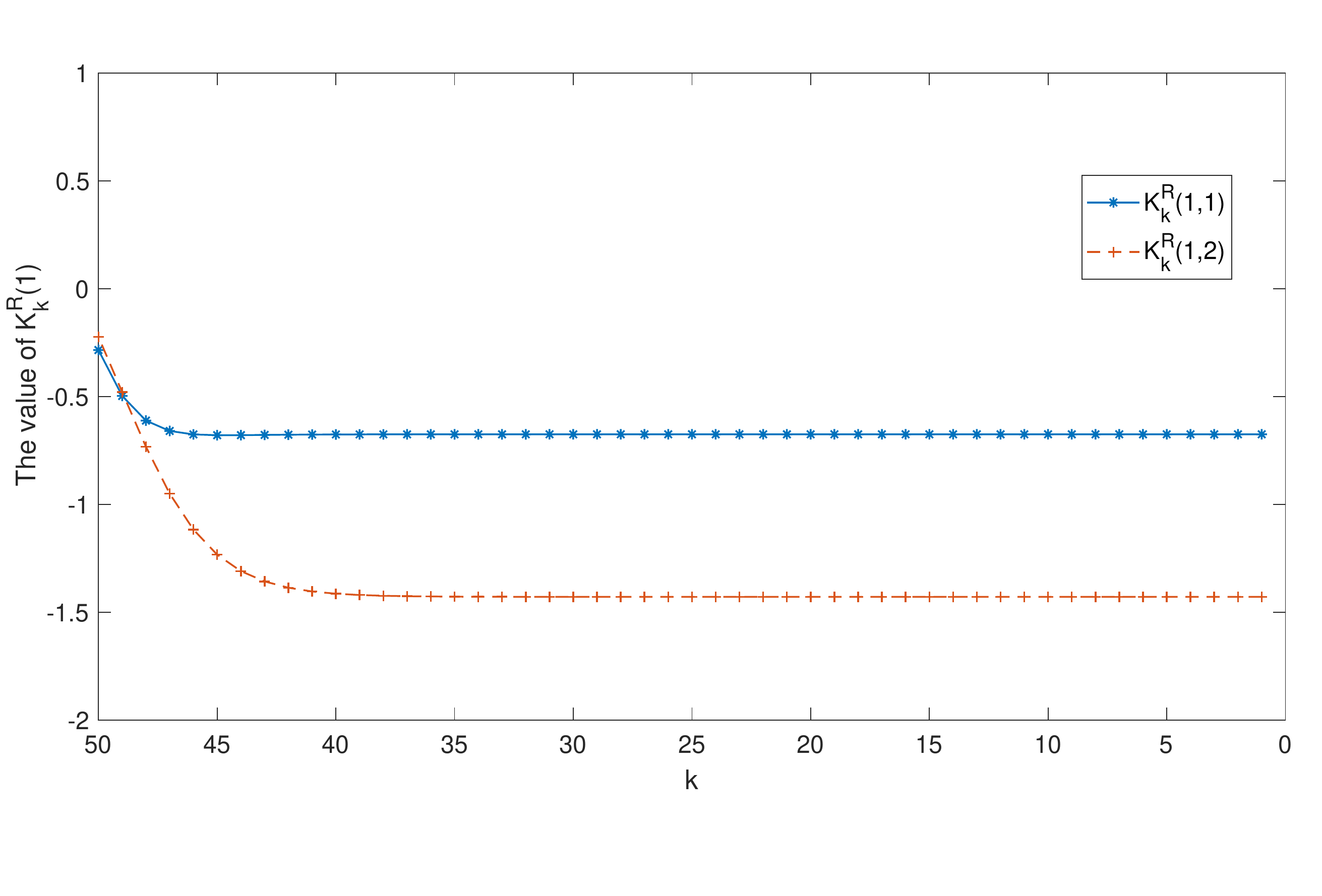}\\
  \caption{Closed-loop Nash equilibrium: $K_k^R(1,i), i=1,2$, the first column value of $K_k^R$.}\label{Figure4}
\end{figure}
\begin{figure}
  \centering
  \includegraphics[width=0.38\textwidth]{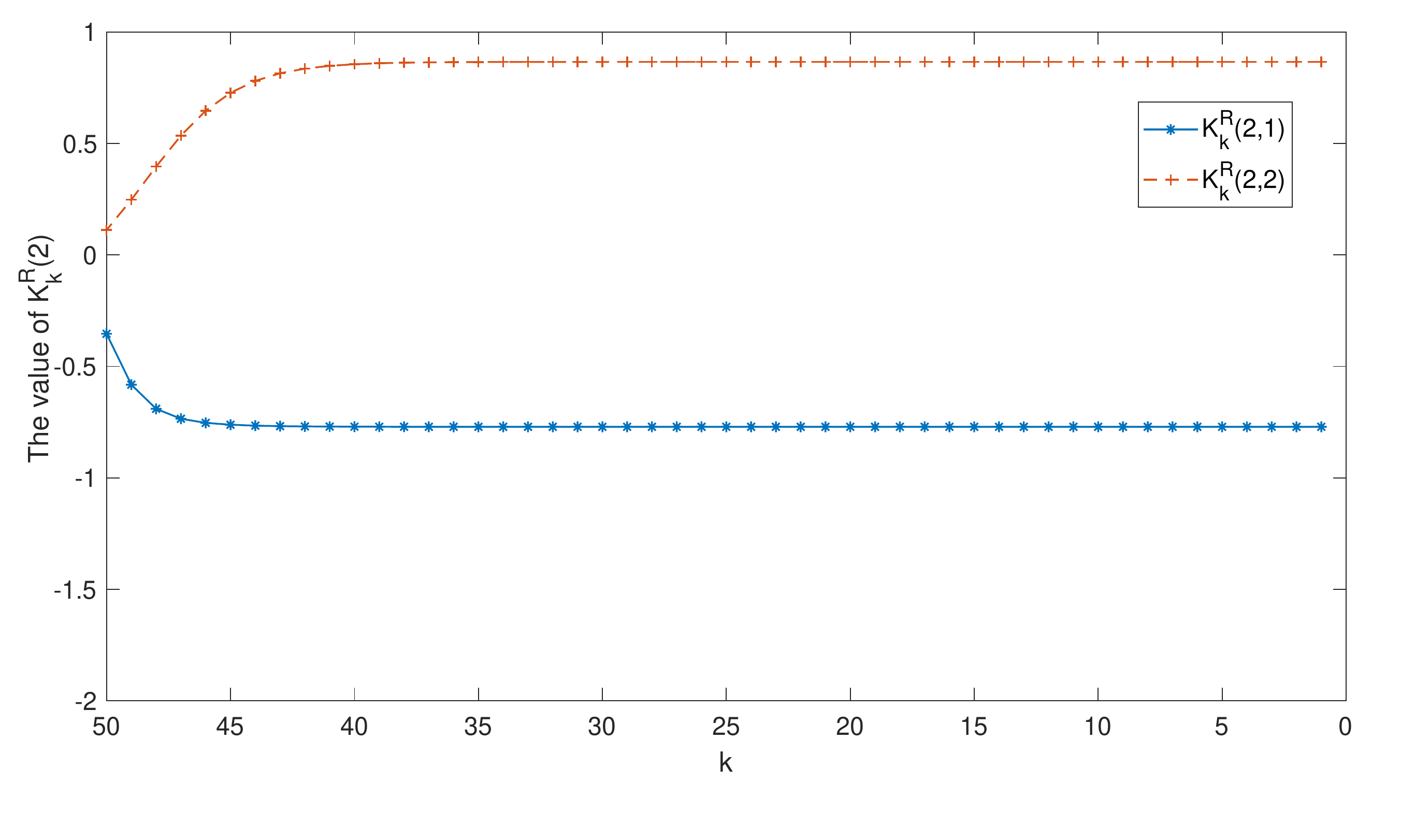}\\
  \caption{Closed-loop Nash equilibrium: $K_k^R(2,i), i=1,2$, the second column value of $K_k^R$.}\label{Figure5}
\end{figure}

In order to illustrate the main results obtained in Theorem \ref{th-02}, a numerical example is provided in this section.

Without loss of generallity, we shall consider the system dynamics \eqref{ss1} and the cost functions \eqref{cf1}-\eqref{cf2} with the following coefficients:
\begin{align}\label{coes}
&N=50, p=0.5, \mu=0,A= \begin{bmatrix}
{1.2}&{0}\\
{0}&{1.1}
\end{bmatrix};\notag\\
&B^L= \begin{bmatrix}
{0.3}&{0.2}\\
{0.4}&{-0.1}
\end{bmatrix},
B^R= \begin{bmatrix}
{0.1}&{0.2}\\
{0}&{0.1}
\end{bmatrix};\notag\\
&Q^L=Q^R=S^L=S^R=M^L=M^R=I_2;\notag\\
&\Lambda^L=\begin{bmatrix} S^L &  \\
 &  M^L \end{bmatrix}=I_4, P^L_{N+1}=P^R_{N+1}=I_2.
\end{align}

From the above coefficients in \eqref{coes}, obviously, Assumption \ref{ass1} is satisfied. Consequently, $P_k$ can be calculated from \eqref{rere} backwardly using \eqref{coes}, then it can be verified that $\Lambda^L+\mathcal {B}^TP_{k+1}^L\mathcal {B}$ and $S^R+(B^L)^TP_{k+1}^RB^L$ are positive definite for $k=1, \cdots, 50$. Using Theorem \ref{th-02}, we can conclude that the closed-loop Nash equilibrium of Problem LRSNG is unique, and $(K_k^L, K_k^R)$, $k=0, \cdots, 50$ can be calculated from \eqref{rere}, which are shown in Figures \ref{Figure2}-\ref{Figure5}, backwardly.

As can be seen from Figures \ref{Figure2}-\ref{Figure5}, it is apparent that the closed-loop Nash equilibrium $(K_k^L,K_k^R)$ would be convergent with $N$ becomes large.

\section{Conclusion}

In this paper, we have discussed the open-loop and closed-loop local and remote Nash equilibrium for LRSNG problem for discrete-time stochastic systems with inconsistent information structure. This paper extends the existing works on LQ stochastic games with consistent information structure to the inconsistent information structure case. Both the open-loop and closed-loop Nash equilibrium have been derived in this paper, and a numerical example is given to illustrate the main results. We believe the proposed methods and results would shed a light in solving other kinds of stochastic game problem with inconsistent information structure.

\ifCLASSOPTIONcaptionsoff
  \newpage
\fi

\end{document}